\documentclass[12pt]{article}

\usepackage{enumerate}
\usepackage{amssymb,amsmath,amsthm,bbm,enumitem}
\usepackage[a4paper, total={6.3in, 9.2in}]{geometry}
\usepackage{setspace}
\usepackage{hyperref}

\newtheorem{theorem}{Theorem}[section]
\theoremstyle{definition}

\theoremstyle{definition}

\theoremstyle{definition}

\theoremstyle{definition}

\theoremstyle{remark}

\theoremstyle{definition}
\newtheorem{lemma}{Lemma}[section]

\theoremstyle{remark}

\setlist[itemize]{leftmargin=0.4cm,labelindent=\parindent}

\author{\small \begin{tabular}{ccc}
Wendi Han & Guangyue Han\\
The University of Hong Kong & The University of Hong Kong\\
email: wendyhan@connect.hku.hk & email: ghan@hku.hk\\
\end{tabular}}

\date{{\normalsize \today}}

\title{A New Proof of Hopf's Inequality Using a Complex Extension of the Hilbert Metric}

\begin{document} \maketitle

\begin{abstract}
Hopf's inequality for positive linear operators yields a strengthening of Perron's theorem. We give in this paper an alternative proof of this strengthening using a complex extension of the Hilbert metric.
\end{abstract}


\section{Introduction}

Let $n$ be an integer greater than or equal to $2$. Let $A = (a_{ij})$ be an $n \times n$ positive matrix, i.e., $a_{i,j} > 0$ for all $i, j$. By Perron's theorem~\cite{Perron07}, the largest eigenvalue (in modulus) of $A$, denoted by $\rho(A)$, is unique, real and positive, and therefore, the {\em spectral ratio} $\kappa(A)$ of $A$, defined as
$$
\kappa(A) \triangleq \max\{|\lambda|: \lambda \mbox{ is an eigenvalue of } A, \lambda \neq \rho(A)\}/\rho(A),
$$
is strictly less than $1$. Ostrowski~\cite{Ostrowski63} strengthened this result and showed that
\begin{equation}  \label{Ostrowski-Strengthening}
\kappa(A) \leq \frac{M^2-m^2}{M^2+m^2},
\end{equation}
where $m = \min_{i, j} a_{ij}$ and $M = \max_{i, j} a_{ij}$. Inspired by Ostrowski's theorem, Hopf~\cite{Hopf1} further strengthened Perron's theorem and showed that
\begin{equation}  \label{Hopf-Strengthening}
\kappa(A) \leq \frac{M-m}{M+m}.
\end{equation}
It has been observed~\cite{Ostrowski64} that Hopf's strengthening is tight in the sense that there are examples of $A$ for which (\ref{Hopf-Strengthening}) holds with equality.

Though not the major concern of this work, let us mention that Frobenius~\cite{Frobenius0809, Frobenius12} generalized Perron's theorem to non-negative matrices, which is popularly known as the Perron-Frobenius theorem. This result is the key pillar of the theory of non-negative matrices, which has a wide range of applications in multiple disciplines; see, e.g.,~\cite{Seneta, Minc88, Berman94, Bapat97, nonlinear}. Accordingly, there are numerous results characterizing the isolation of the largest eigenvalue of non-negative matrices, most of them in the forms of upper bounds on the modulus of the second largest eigenvalue; see, e.g.,~\cite{Rothblum85} and the references therein. And it is worthwhile to note that for certain special families of symmetric non-negative matrices (such as adjacency matrices of a regular graph and transition probabilities matrices of a reversible stationary Markov chain), numerous Cheeger-type inequalities, which are in the forms of bounds on the difference between the largest and second largest eigenvalue, have been established; see, e.g.~\cite{Chung97, Brouwer12, Montenegro, Levin} and references therein.

Although it often shows up in the literature, the exact expression as in (\ref{Hopf-Strengthening}) actually does not appear in~\cite{Hopf1} and only follows from Theorem $4$ therein, stated for more general positive linear operators. As a matter of fact, a careful examination of the proof of Theorem $4$ reveals that it yields a bound stronger than (\ref{Hopf-Strengthening}).

To precisely state this stronger result, we need to introduce some notation and terminologies. Let $W$ denote the standard simplex in the $n$-dimensional Euclidean space:
\begin{equation} \label{real_simplex}
W = \left\{w=(w_1, w_2, ..., w_n) \in \mathbb{R}^n: \sum_{i=1}^n w_i =1, \;\; w_i \geq 0 \mbox{ for all } i \right\},
\end{equation}
and let $W^\circ$ denote its interior, consisting of all the positive vectors in $W$. Let $d_H$ denote the {\em Hilbert metric} on $W^{\circ}$, which is defined~\footnote{The Hilbert metric is often defined on a projective space (see, e.g.,~\cite{Seneta, nonlinear}), which is equivalent to the definition in this paper up to a usual normalization.} by
\begin{equation} \label{Hilbert-Metric}
d_H (v, w) \triangleq \max_{i,j} \log \left(\frac{w_i/w_j}{v_i/ v_j}\right), \mbox{ for any two vectors } v, w \in W^{\circ}.
\end{equation}
For any positive vector $w=(w_1, w_2, \dots, w_n) \in \mathbb{R}^n$, we define its normalized version $\mathcal{N}(w)$ as
\begin{equation} \label{N}
\mathcal{N}(w) = \frac{(w_1, w_2, \dots, w_n)}{w_1+w_2+\dots+w_n},
\end{equation}
which obviously belongs to $W^{\circ}$. Apparently, the matrix $A$ induces a mapping $f_A: W^{\circ} \to W^{\circ}$, defined by
\begin{equation} \label{fA}
f_A(w) = \mathcal{N}(A w), \mbox{ for any vector } w \in W^{\circ}.
\end{equation}
It is well known that $f_A$ is a contraction mapping under the Hilbert metric and the contraction coefficient $\tau(A)$, defined by
$$
\tau(A) \triangleq \sup_{v \neq w \in W^{\circ}} \frac{d_H (Av, Aw)}{d_H (v, w)}
$$
and often referred to as the \textit{Birkhoff contraction coefficient}, can be explicitly computed as
\begin{equation}
\tau(A) = \frac{1-\sqrt{\phi(A)}}{1+\sqrt{\phi(A)}},
\end{equation}
where
\begin{equation}
\phi(A) = \min_{i,j,k,l}\frac{a_{i k}a_{j l}}{a_{j k}a_{i l}}.
\end{equation}

We are now ready to state the aforementioned stronger result:
\begin{theorem} \label{main-theorem}
For an $n \times n$ positive matrix $A$, we have
\begin{equation} \label{main-formula}
\kappa(A) \leq \tau(A).
\end{equation}
\end{theorem}
\noindent As mentioned before, Theorem~\ref{main-theorem} follows from Theorem $4$ in~\cite{Hopf1}, which is a contraction result with respect to the Hopf oscillation. Ostrowski~\cite{Ostrowski64} modified Birkhoff's argument in~\cite{Birkhoff} and gave an alternative proof of Theorem~\ref{main-theorem}, which however still used the Hopf oscillation. In this work, we will give a new proof of Theorem~\ref{main-theorem} using a complex extension of the Hilbert metric in lieu of the Hopf oscillation. As it turned out, the complex Hilbert metric can be applied elsewhere; more specifically, it has been used~\cite{HH} to establish the analyticity of entropy rate of hidden Markov chains and specify the corresponding domain of analyticity.

\section{A Complex Hilbert Metric} \label{CHM}

Let $W_{\mathbb{C}} = \{w=(w_1, w_2, \dots, w_n) \in \mathbb{C}^n : \sum_{i=1}^n w_i=1\}$ and let $W_{\mathbb{C}}^+ = \{w=(w_1, w_2, \dots, w_n) \in W_{\mathbb{C}}: \mathcal{R}(w_i/w_j) >0 \mbox{ for all } i,j \}$. The following complex extension of the Hilbert metric has been proposed in~\cite{HH}:
\begin{equation}  \label{ComplexHilbertMetric}
d_H(v, w)=\max_{i, j} \left| \log \left( \frac{w_i/w_j}{v_i/v_j} \right) \right|, \mbox{ for any }
v, w \in W_{\mathbb{C}}^+,
\end{equation}
where $\log(\cdot)$ is taken as the principal branch of the complex $\log(\cdot)$ function. Here we remark that there are other complex extensions of the Hilbert metric; see, e.g.,~\cite{Rugh, Dubois}. Our treatment however only uses the extension in (\ref{ComplexHilbertMetric}), which will henceforth be referred to as the complex Hilbert metric. For any $\varepsilon > 0$, we define
{\small \begin{equation} \label{W-Neighborhood}
\hspace{-0.2cm} W_{\mathbb{C}}^{\circ}(\varepsilon) \triangleq \{w=(w_1, w_2, \cdots, w_n) \in W_{\mathbb{C}}: \exists \, v \in W^{\circ} \mbox{ such that } |w_i-v_i| \leq \varepsilon v_i \mbox{ for all } i\}.
\end{equation}}
\!\!\!It can be easily verified that for $\varepsilon$ small enough, $W_{\mathbb{C}}^{\circ}(\varepsilon) \subset W_{\mathbb{C}}^+$ and thereby the complex Hilbert metric is well-defined on $W_{\mathbb{C}}^{\circ}(\varepsilon)$.

Extending the definition in (\ref{N}), for any complex vector $w=(w_1, w_2, \dots, w_n)$ with $w_1+w_2+\dots+w_n \neq 0$, we define its normalized version $\mathcal{N}(w)$ as
$$
\mathcal{N}(w) = \frac{(w_1, w_2, \dots, w_n)}{w_1+w_2+\dots+w_n},
$$
which obviously belongs to $W_{\mathbb{C}}$. And furthermore, for any $\varepsilon > 0$, extending the definition in (\ref{fA}), we define $f_A: W_{\mathbb{C}}^{\circ}(\varepsilon) \to W_{\mathbb{C}}^{\circ}(\varepsilon)$ by:
\begin{equation} \label{complex-fA}
f_A(w) = \mathcal{N}(A w), \mbox{ for any vector } w \in W_{\mathbb{C}}^{\circ}(\varepsilon),
\end{equation}
which is well-defined if $\varepsilon$ is small enough.

The following lemma has been implicitly established in~\cite{HH}. We outline its proof for completeness and clarity. An interested reader may refer to the proofs of Theorem $2.4$ in~\cite{HH} and relevant lemmas for more technical details.
\begin{lemma} \label{complex-contraction}
Consider an $n \times n$ positive square matrix $A$. For any small enough $\varepsilon > 0$, there exists $0 < \tau_{\varepsilon}(A) <1$ such that for any $x, y \in W_{\mathbb{C}}^{\circ}(\varepsilon)$,
\begin{equation} \label{still-contraction}
d_H(f_A(x), f_A(y)) \leq \tau_{\varepsilon}(A) d_H(x, y),
\end{equation}
and moreover, $\tau_{\varepsilon}(A)$ tends to $\tau(A)$ as $\varepsilon$ tends to $0$.
\end{lemma}

\begin{proof}
First of all, we note, by the definition in (\ref{ComplexHilbertMetric}), that for any $x, y \in W_{\mathbb{C}}^{\circ}(\varepsilon)$,
$$
\frac{d_H(f_A(x), f_A(y))}{d_H(x,y)}=\frac{d_H(\mathcal{N}(A x), \mathcal{N}(A y))}{d_H(x,y)}=\max_{i,j}|L_{i,j}|,
$$
where
$$
L_{i, j} = \frac{\log \left(\sum_m a_{i m} x_m /\sum_m a_{j m}  x_m \right)-\log \left(\sum_m a_{i m} y_m /\sum_m a_{j m} y_m \right)}{\max_{k,l}|\log(x_k/y_k) - \log(x_l/y_l)|}.
$$
Letting $c_i=\log(x_i/y_i)$ for all $i$ and choosing $p, q$ such that
$|c_p -c_q| = \max_{k,l}|c_k - c_l|$, we note that $L_{i, j}$ can be rewritten as
$$
L_{i,j} = \frac{\log \left(\sum_m e^{c_m-c_q} a_{i m} y_m/\sum_m e^{c_m-c_q} a_{j m} y_m\right)-\log \left(\sum_m a_{i m} y_m/\sum_m a_{j m} y_m \right)}{|c_p -c_q|}.
$$
An application of the mean value theorem then yields that there exists $\xi \in [0, 1]$ such that
$$
|L_{i,j}| \leq \sum_l \frac{c_l - c_q}{|c_p-c_q|}\left(\frac{e^{(c_l-c_q)\xi} a_{i l} y_l}{\sum_m e^{(c_m-c_q)\xi} a_{i m} y_m}-\frac{e^{(c_l-c_q)\xi} a_{j l} y_l}{\sum_m e^{(c_m-c_q)\xi} a_{j m} y_m}\right).
$$
By the definition of $W_{\mathbb{C}}^{\circ}(\varepsilon)$, there exist $x^{\circ}, y^{\circ} \in W^\circ$ such that for some constant $C_1 > 0$,
$$
|x_k-x^{\circ}_k| \leq C_1 \varepsilon x^{\circ}_k, \quad |y_k-y^{\circ}_k| \leq C_1 \varepsilon y^{\circ}_k \mbox{ for all } k.
$$
Now, let
$$
D_l = \frac{e^{(c_l-c_q)\xi} a_{i l} y_l}{\sum_m e^{(c_m-c_q)\xi} a_{i m} y_m}-\frac{e^{(c_l-c_q)\xi} a_{j l} y_l}{\sum_m e^{(c_m-c_q)\xi} a_{j m} y_m},
$$
and
$$
D^{\circ}_l = \frac{e^{(c^{\circ}_l-c^{\circ}_q)\xi} a_{i l} y^{\circ}_l}{\sum_m e^{(c^{\circ}_m-c^{\circ}_q)\xi} a_{i m} y^{\circ}_m}-\frac{ e^{(c^{\circ}_l-c^{\circ}_q)\xi} a_{j l} y^{\circ}_l}{\sum_m e^{(c^{\circ}_m-c^{\circ}_q)\xi} a_{j m} y^{\circ}_m},
$$
where we have, similarly as above, defined $c^{\circ}_i = \log(x^{\circ}_i/y^{\circ}_i)$ for all $i$. It then follows from the established facts that for some constant $C_2 > 0$,
\begin{equation*}
\left|\sum_l \frac{c_l-c_q}{|c_p-c_q|} D_l -\sum_l \frac{c_l-c_q}{|c_p-c_q|} D^{\circ}_l \right|< C_2 C_1 \varepsilon,
\end{equation*}
and
$$
\left|\sum_l \frac{c_l-c_q}{|c_p-c_q|} D^{\circ}_l\right| \leq \tau (A)
$$
that
$$
\left|\sum_l \frac{c_l-c_q}{|c_p-c_q|} D_l\right| \leq C_2 C_1 \varepsilon + \tau (A),
$$
which immediately implies that
$$
\frac{d_H(f_A(x),f_A(y))}{d_H(x,y)} \leq  C_2 C_1 \varepsilon + \tau (A).
$$
Setting $\tau_{\varepsilon}(A) = C_2 C_1 \varepsilon + \tau (A)$ and noting that $\varepsilon$ can be chosen arbitrarily small, we establish (\ref{still-contraction}) and conclude that $\tau_{\varepsilon}(A)$ tends to $\tau(A)$ as $\varepsilon$ tends to $0$.
\end{proof}

\section{Proof of Theorem~\ref{main-theorem}}  \label{main-proof}

For a subset $S$ of $W^{\circ}$, we generalize the definition in (\ref{W-Neighborhood}) and define
$$
\hspace{-0.3cm} S_{\mathbb{C}}(\varepsilon) \triangleq \{w=(w_1, w_2, \cdots, w_n) \in W_{\mathbb{C}}: \exists \; v \in S \mbox{ such that } |w_i-v_i| \leq \varepsilon v_i \mbox{ for all } i \}.
$$
We will need the following lemma, which, roughly speaking, asserts the equivalence between the Euclidean metric (denoted by $d_E$) and the Hilbert metric on a complex neighborhood of a compact subset of $W^{\circ}$
\begin{lemma} \label{dEdH}
For any compact subset $S$ of $W^{\circ}$, there exists $\varepsilon_0 > 0$ such that there exist constants $G_1, G_2 > 0$ such that for all $0 < \varepsilon < \varepsilon_0$ and for all $v, w \in S_{\mathbb{C}}(\varepsilon)$,
$$
G_1 d_H(v, w) < d_E(v, w)< G_2 d_H(v, w).
$$
\end{lemma}

\begin{proof}
The lemma follows from some straightforward arguments underpinned by the mean value theorem and the compactness of $S$, which are completely parallel to those in the proof of Proposition $2.1$ in~\cite{H} (a real version of this lemma).
\end{proof}

We are now ready for the proof of Theorem~\ref{main-theorem}.
\begin{proof}
Consider an $n \times n$ positive square matrix $A$. Let $x=(x_1, x_2, \dots, x_n)$ be the eigenvector corresponding to $\rho(A)$. By the Perron-Frobenius theorem, we can choose $x$ to be a positive vector with $x_1+x_2+\dots+x_n=1$, i.e., $x \in W^{\circ}$. Let $\lambda$ be an eigenvalue of $A$ that is different from $\rho(A)$ and let $y$ be a corresponding eigenvector. Here we remark that while $\rho(A)$ and $x$ are real, $\lambda$ and $y$ can be complex.

Now, consider a compact subset $S$ of $W^{\circ}$ that contains $x$. It can be easily verified that for any $\varepsilon > 0$, there exists $n_0 \in \mathbb{N}$ such that for any $n \geq n_0$,
$$
\mathcal{N}(A^n(x+y))=\mathcal{N}(\rho^n(A) x + \lambda^n y) \in S_{\mathbb{C}}(\varepsilon).
$$
Henceforth, we let $v=\rho(A)^{n_0} x$ and $w = \lambda^{n_0} y$. For any $m \in \mathbb{N}$, it can be verified that
\begin{align*}
d_H(\mathcal{N}(A^m v), \mathcal{N}(A^m(v+w))) & = d_H(\mathcal{N}(\rho(A)^m v), \mathcal{N}(\rho(A)^m v + \lambda^m w))\\
                                               & = d_H(\mathcal{N}(v), \mathcal{N}(v + \tilde{\lambda}^m w)),
\end{align*}
where we have written $\lambda/\rho(A)$ as $\tilde{\lambda}$ for notational simplicity. Now, using the definition of the complex Hilbert metric, we continue
\begin{align}
d_H(\mathcal{N}(A^m v), \mathcal{N}(A^m(v+w))) &=\max_{i,j = 1,2,...,n} \left|\log \frac{({v_{i}+\tilde{\lambda}^m w_{i}})/({v_{j}+\tilde{\lambda}^m w_{j}})}{{v_{i}}/{v_{j}}}\right| \nonumber\\
&= \max_{i,j = 1,2,...,n} \left|\log \frac{{1+\tilde{\lambda}^m (w_{i}}/{v_{i}})}{{1+\tilde{\lambda}^m (w_{j}}/{v_{j}})}\right| \nonumber\\
&= \max_{i,j = 1,2,...,n} \left|\log \left(1+{\frac{\tilde{\lambda}^m (w_{i}/v_{i})-(w_{j}/v_{j})}{1+ \tilde{\lambda}^m (w_{j}/v_{j})}}\right)\right| \nonumber\\
&= \max_{i,j = 1,2,...,n} \left|\log\left(1+{\frac{(w_{i}/v_{i})-(w_{j}/v_{j})}{(1/\tilde{\lambda}^m)+(w_{j}/v_{j})}}\right)\right| \nonumber\\
&= \left|\log\left(1+{\frac{(w_{i_0}/v_{i_0})-(w_{j_0}/v_{j_0})}{(1/\tilde{\lambda}^m)+(w_{j_0}/v_{j_0})}}\right)\right|, \label{many-equalities}
\end{align}
where we have assumed $i_0, j_0$ achieve the maxima in (\ref{many-equalities}). We note that $w_{i_0}/v_{i_0} \neq w_{j_0}/v_{j_0}$, since otherwise it would mean $d_H(\mathcal{N}(A^m v), \mathcal{N}(A^m(v+w)))=0$ and therefore $w$ would be a scaled version of $v$, contradicting the fact that $\lambda$ is different from $\rho(A)$.

It follows from the fact that $0 < \tilde{\lambda} < 1$ that there exists a constant $C_1 > 0$ such that for all $m$,
$$
d_H(\mathcal{N}(A^m v), \mathcal{N}(A^m(v+w))) = \left|\log\left(1+{\frac{(w_{i_0}/v_{i_0})-(w_{j_0}/v_{j_0})}{(1/\tilde{\lambda}^m)+(w_{j_0}/v_{j_0})}}\right)\right| \geq C_1 \left|{\frac{(w_{i_0}/v_{i_0})-(w_{j_0}/v_{j_0})}{(1/\tilde{\lambda}^m)+(w_{j_0}/v_{j_0})}}\right|.
$$
And by Lemmas~\ref{complex-contraction} and~\ref{dEdH}, there exist $0 < \tau_{\varepsilon}(A) < 1$ and a constant $C_2 > 0$ such that
$$
d_H (\mathcal{N}(A^m v), \mathcal{N}(A^m(v+w))) \leq C_2 \tau_{\varepsilon}^m(A) d_E(\mathcal{N}(v), \mathcal{N}(v+w)),
$$
which immediately implies that
$$
C_1 \left|{\frac{1}{(1/\tilde{\lambda}^m)+(w_{j_0}/v_{j_0})}}\right| \leq C_2 \tau_{\varepsilon}^m(A) \frac{d_E(\mathcal{N}(v), \mathcal{N}(v+w))}{|(w_{i_0}/v_{i_0})-(w_{j_0}/v_{j_0})|}.
$$
One then verifies that there exists a constant $C_3 > 0$ (which depends only on $x, y$) such that
$$
\frac{d_E(\mathcal{N}(v), \mathcal{N}(v+w))}{|(w_{i_0}/v_{i_0})-(w_{j_0}/v_{j_0})|} < C_3,
$$
and furthermore, there exists a constant $C_4 > 0$ such that for all $m$,
$$
\left|{\frac{1}{(1/\tilde{\lambda}^m)+(w_{j_0}/v_{j_0})}}\right| \geq C_4 \tilde{\lambda}^m.
$$
It then follows that after choosing $\varepsilon$ small enough and then $n_0$ large enough, we have
$$
C_1 C_4 \tilde{\lambda}^m \leq C_2 C_3 \tau_{\varepsilon}^m(A),
$$
which, upon letting $m$ tend to infinity, yields $\tilde{\lambda} \leq \tau_{\varepsilon}(A)$, where we have used the fact that all the constants $C_1, C_2, C_3, C_4$ can be chosen independent of $\varepsilon$. Moreover, using the fact that $\varepsilon$ can be chosen arbitrarily small, we apply Lemma~\ref{complex-contraction} to obtain $\tilde{\lambda} \leq \tau(A)$, which immediately leads to
$\kappa(A) \leq \tau(A)$, as desired.
\end{proof}

\bigskip
{\bf Acknowledgement.} This work is supported by the Research Grants Council of the Hong Kong Special Administrative Region, China, under Project 17301017 and by the National Natural Science Foundation of China, under Project 61871343.


\begin{thebibliography}{9}

\bibitem{Bapat97}
R.~Bapat and T.~RagHavan.
{\em Nonnegative Matrices and Applications}, New York: Cambridge University Press, 1997.

\bibitem{Berman94}
A.~Berman and R.~Plemmons.
{\em Nonnegative Matrices in the Mathematical Sciences}, Philadephia, Pa.: Society for Industrial and Applied Mathematics, 1994.

\bibitem{Birkhoff}
G.~Birkhoff.
\newblock Extensions of Jentzsch's Theorem. \newblock {\em Transactions of the American Mathematical Society},
vol. 85, no. 1, pp. 219-227, 1957.

\bibitem{Brouwer12}
A.~Brouwer and W.~Haemers. {\em Spectra of graphs}, Springer, New York, 2012.

\bibitem{Chung97}
F.~Chung. {\em Spectral graph theory}, Providence, R.I.: Published for the Conference Board of the mathematical sciences by the American Mathematical Society, 1997.

\bibitem{Dubois}
L.~Dubois.
\newblock Projective metrics and contraction principles for complex cones.
\newblock {\em Journal of the London Mathematical Society}, vol. 79, no. 3, pp. 719-727, 2009.

\bibitem{H}
G.~Han and B.~Marcus.
\newblock Analyticity of entropy rate of hidden Markov chains.
\textit{IEEE Trans. Info. Theory}, vol. 52, no. 12, pp. 5251-5266, 2006.

\bibitem{HH}
G.~Han, B.~Marcus and Y.~Peres. A note on a complex Hilbert metric with application to domain of analyticity for entropy rate of hidden Markov processes. {\em Entropy of Hidden Markov Processes and Connections to Dynamical Systems}, London Mathematical Society Lecture Note Series, vol. 385, pp. 98-116, 2011.

\bibitem{Frobenius0809}
G.~Frobenius.
\"{U}ber matrizen aus positiven elementen. {\em Sitzungsberichte Preussische Akademie der Wissenschaft}, Berlin, pp. 471–476, 514–518, 1908, 1909.

\bibitem{Frobenius12}
G.~Frobenius.
\"{U}ber matrizen aus nicht negativen elementen. {\em Sitzungsberichte Preussische
Akademie der Wissenschaft}, Berlin, pp. 456–477, 1912.


\bibitem{Hopf1}
E.~Hopf.
An inequality for positive linear integral operators.  {\em J. Math. Mech.}, vol. 12, no. 5, pp. 683–692, 1963.


\bibitem{nonlinear}
B.~Lemmens and R.~Nussbaum.
\newblock {\em Nonlinear Perron-Frobenius Theory}, Cambridge University Press, 2012.

\bibitem{Levin}
D.~Levin and Y.~Peres.
\newblock {\em Markov Chains and Mixing Times}, American Mathematical Society, 2nd Revised Edition, 2017.

\bibitem{Minc88}
H.~Minc.
{\em Nonnegative Matrices}, New York: Wiley, 1988.

\bibitem{Montenegro}
R.~Montenegro and P.~Tetali.
\newblock {\em Mathematical Aspects of Mixing Times in Markov chains}, Foundations and Trends in
Theoretical Computer Science, Now Publishers, 2006.

\bibitem{Ostrowski63}
A.~Ostrowski.
On positive matrices. {\em Math. Ann.}, vol. 150, no. 3, pp. 276–284, 1963.

\bibitem{Ostrowski64}
A.~Ostrowski.
Positive matrices and functional analysis. {\em Recent Advances in Matrix Theory}, Madison: Univ. of Wisconsin Press, 1964.

\bibitem{Perron07}
O.~Perron. Grundlagen f\"{u}r eine theorie des Jacobischen Kettenbruchalgorithmus. {\em Math. Ann.}, vol. 64, pp. 11–76, 1907.

\bibitem{Rothblum85}
U.~Rothblum and C.~Tan.
Upper bounds on the maximum modulus of subdominant eigenvalues of nonnegative matrices. {\em Linear Algebra Appl}, vol. 66, pp. 45-86, 1985.

\bibitem{Rugh}
H.~Rugh.
\newblock Cones and gauges in complex spaces: Spectral gaps and complex Perron-Frobenius theory.
\newblock {\em Annals of Mathematics}, vol. 171, no. 3, 2010.

\bibitem{Seneta}
E.~Seneta.
\newblock {\em Non-negative Matrices and Markov Chains}, Springer Series in Statistics,
Springer-Verlag, New York Heidelberg Berlin, 1980.
\end{thebibliography}
\end{document}